\documentclass[11pt]{article}
\usepackage{amsmath,amssymb}
\usepackage{umlaute}
\usepackage{graphicx,epsfig}
\usepackage{psfrag}
\usepackage{amsfonts}
\usepackage{latexsym}
\usepackage{pstricks}
\usepackage{bbm}
\newtheorem{theorem}{\bf Theorem}
\newtheorem{lemma}[theorem]{\bf Lemma}

\newenvironment{proof}{\noindent{\bf Proof}}

\newcommand{\CE}{\mathcal{E}}

\newcommand{\CT}{\mathcal{T}}
\newcommand{\CN}{\mathcal{N}}
\newcommand{\CV}{\mathcal{V}}

\newcommand{\bV}{\mbox{\boldmath{$V$}}}
\newcommand{\bW}{\mbox{\boldmath{$W$}}}

\newcommand{\fb}{\mbox{\boldmath{$f$}}}

\newcommand{\bu}{\mbox{\boldmath{$u$}}}

\newcommand{\bv}{\mbox{\boldmath{$v$}}}

\newcommand{\bx}{\mbox{\boldmath{$x$}}}

\newcommand{\bzero}{\mbox{$\bf 0$}}
\newcommand{\BBR}{\mbox{$\mathbb{R}$}}
\newcommand{\bN}{\mbox{$\mathbb{N}$}}
\newfont{\twelvemsb}{msbm10 at 11.6pt}

\renewcommand{\div}{\mathop{\rm div\,}}
\newcommand{\divh}{{\rm div_h}}

\begin{document}

\title{A nonconforming finite element method for 
the Stokes equations using 
the Crouzeix-Raviart element for the velocity and the 
standard linear element for the pressure}

\author{Bishnu P.~Lamichhane\thanks{School of Mathematical \& Physical Sciences,
Mathematics Building,
University of Newcastle,
University Drive,
Callaghan, NSW 2308, Australia, {\tt Bishnu.Lamichhane@newcastle.edu.au}}}

\maketitle

\begin{abstract}
We present a finite element method for Stokes equations using 
the Crouzeix-Raviart element for the velocity and the continuous linear
element for the pressure. We show that the inf-sup condition is 
satisfied for this pair. Two numerical experiments 
are presented to support the theoretical results. 
\end{abstract}


 
\section{Introduction}
The finite element method is very popular method for approximating 
the solutions of partial differential equations. 
There are many finite element methods for Stokes and Navier-Stokes equations leading 
to optimal convergence \cite{GR86,BF91,BS94}.  However, the search for 
simple and efficient finite element schemes for Stokes and
Navier-Stokes equations is still an active area of research. 
The nonconforming techniques and discontinuous Galerking methods 
 have also gained high popularity. In particular, it is easier to
 prove the inf-sup condition with the nonconforming methods , and 
they are also quite simple to implement. Moreover, nonconforming
finite element basis functions have smaller support compared to
conforming finite elements. One of the most popular nonconforming
finite element methods is the method based on discretizing the
velocity with the lowest order Crouzeix-Raviart element and the pressure with the 
piecewise constant functions \cite{CR73}. 
Recently a stabilized finite element method is presented using 
again Crouzeix-Raviart element for discretizing the velocity but 
replacing the piecewise constant pressure element with the 
continuous piecewise linear pressure \cite{LC08}. 
We want to emphasize two benefits of using the continuous linear pressure over the 
discontinuous piecewise constant pressure. With 
the continuous pressure we expect to get a 
better approximation for the pressure. Although 
we could not prove this better approximation theoretically,
numerical experiments support this argument. 
The second advantage is that it is easy to 
visualize the continuous pressure than the discontinuous pressure.

In this contribution, we show that the stabilization proposed in
\cite{LC08} is unnecessary. In fact, we show that the finite element
pair - the lowest order Crouzeix-Raviart element for the velocity and the continuous linear
 element for the pressure - yields a stable approximation scheme. 
The proof is based on using an interpolation operator satisfying
certain conditions. The proof also establishes that 
for a low order finite element method based on simplicial meshes with 
a piecewise constant pressure, the piecewise constant pressure 
can be replaced by the continuous linear pressure. 
The rest of the paper is organized as follows. In the next section we
recall the Stokes equations, and we prove our main results in Section
\ref{sec:fe}.  Two numerical experiments are presented 
in \ref{sec:num}  to support the theoretical result. 
Finally, a conclusion is drawn in Section
\ref{sec:con}.

\section{Stokes equations}\label{sec:bvp}
This section is devoted to the introduction of the boundary 
value problem of the Stokes equations.  Let $\Omega$ in $\BBR^d$, $d\in \{2,3\}$, be a 
bounded domain with polygonal or polyhedral boundary $\Gamma$. 
For a prescribed body force $\fb \in [L^2(\Omega)]^d$, the Stokes 
equations with homogeneous Dirichlet boundary condition  in $\Gamma$ reads
\begin{equation}
\begin{array}{ccc}
-\nu\Delta \bu + \nabla p &=& \fb\quad\text{in}\quad \Omega \\
\div  \bu &=& 0 \quad\text{in}\quad \Omega 
\end{array}
\end{equation}
with $\bu=\bzero$ on $\Gamma$, 
where $\bu$ is the velocity, $p$ is the pressure, and $\nu$ denotes
the viscosity of the fluid. 

Here we use standard notations $L^2(\Omega)$, $H^1(\Omega)$ and 
$H^1_0(\Omega)$ for Sobolev spaces, see \cite{BS94,Cia78} for details. 
Let $\bV :=[H^1_0(\Omega)]^d$  be the vector Sobolev space with 
 inner product $(\cdot,\cdot)_{1,\Omega}$ and norm $\|\cdot\|_{1,\Omega}$
defined in the standard way:  $(\bu,\bv)_{1,\Omega} := \sum_{i=1}^d
(u_i,v_i)_{1,\Omega}$, and the norm being induced by this inner
product. We also define another subspace $M$ of $L^2(\Omega)$ as 
\[ M = \{ q \in L^2(\Omega): \, \int_{\Omega} q \, dx =0\}.\]

The weak formulation of the Stokes equations is to find $(\bu,p) \in   
\bV\times M$ such that 
\begin{equation} \label{stokesw}
\begin{array}{ccc}
\nu \int_{\Omega} \nabla\bu: \nabla \bv\,dx  &+
\int_{\Omega} \div v \,p\,dx &= \ell(\bv),\quad \bv \in \bV,\\
\int_{\Omega} \div \bu\, q \,dx &&= 0,\quad q \in M,
\end{array}
\end{equation}
where $ \ell(\bv) = \int_{\Omega} \fb\cdot \bv\,dx.$
It is well-known that the weak formulation of the Stokes problem is
well-posed \cite{GR86}. In fact, if the domain $\Omega$ is convex, and 
$\fb \in [L^2(\Omega)]^d$, we 
have $\bu \in [H^2(\Omega)]^d$, $p \in H^1(\Omega)$ and 
the a priori estimate holds 
\[ \|\bu \|_{2,\Omega}  + \|p \|_{1,\Omega}  \leq C \|\fb \|_{0,\Omega},\]
where the constant $C$ depends on the domain $\Omega$.

\section{Finite element discretizations}\label{sec:fe}
We consider a quasi-uniform triangulation $\CT_h$ of the 
polygonal or polyhedral domain $\Omega$, where $\CT_h$
consists of simplices, either
triangles or tetrahedra, where $h$ denotes the mesh-size. 
 Note that $\CT_h$ denotes the set
of elements.  For an element $T \in \CT_h$, let $P_n(T)$ be the 
set of all polynominals of degree less than or equal to $ n \in
\bN\cup \{0\}$.  The the set of all 
vertices in $\CT_h$ is denoted by 
$\CV_h:=\{\bx_i\}_{i=1}^N$, and $\CN_h :=\{1,2,\cdots,N\}.$

Let $\CE_h$ be the set of all edges in two
dimensions and faces in three dimensions.
The continuity of a function $\bv_h$ across an edge or a face $e\in \CE_h$ for 
the nonconforming finite element will be enforced according to 
\[ J_e(\bv_h) := \int_{e}[\bv_h]_e \,d\sigma =0,\]
where $[\bv_h]_e$ is the jump of the function $\bv_h$ across the edge
$e$. 
Then the Crouzeix-Raviart
finite element space $\bV_h$ is defined as 
\[ \bW_h := \{\bv_h \in [L^2(\Omega)]^d:\, \bv_h|_{K} \in [P_1(K)]^d,
\; 
K \in \CT_h, \; J_e(\bv_h) =0,\; e \in \CE_h\}.\]
The finite element basis functions of $\bW_h$ are associated 
with the mid-points of the edges of triangles 
or the bary-centers of the faces of  tetrahedra. 
To impose the homogeneous Dirichlet boundary condition on $\Gamma$ we 
 define $\bV_h$ as a subset of $\bW_h$ where 
\[ \bV_h : = \{ \bv_h \in \bW_h:\, \int_{e} \bv_h\,d\sigma =0,
e \in \CE_h \cap \Gamma\}.\]

As $\bV_h \not\subset \bV$, we cannot use 
the standard $H^1$-norm for an element in $\bV_h$. 
So we define a broken $H^1$- norm on $\bV_h$ as 
\[ \|\bv_h\|_{1,h} : =\sqrt{ \sum_{T \in \CT_h}
  \|\bv_h\|^2_{1,T}},\quad \bv_h \in \bV_h.
\]
The broken $H^1$- semi-norm is similarly 
defined 
\[ |\bv_h|_{1,h} : =\sqrt{ \sum_{T \in \CT_h}
  |\bv_h|^2_{1,T}},\quad \bv_h \in \bV_h.
\]
We also define two relevant finite element spaces 
\[ Q_h :=\left\{ q_h \in L^2(\Omega):\, 
q_h|_{K} = P_0(K), \; K \in \CT_h,\; \int_{\Omega} q_h\,dx =0\right\},\]
and 
\[ P_h :=\left\{ q_h \in H^1(\Omega):\, 
q_h|_{K} = P_1(K), \; K \in \CT_h,\; \int_{\Omega} q_h\,dx =0\right\}.\]
Now our discrete weak formulation of Stokes equations can be written
as:  find $(\bu_h,p_h) \in \bV_h \times P_h$ such that
\begin{equation}\label{mixed}
\begin{array}{lllllllll}
a(\bu_h,\bv_h)&+&b(\bv_h,p_h)&=&\ell(\bv_h),\quad &\bv_h &\in& \bV_h , \\
b(\bu_h,q_h)&&&=&0,\quad &q_h &\in& P_h,
\end{array}
\end{equation}
where
\begin{eqnarray*}
a(\bu_h,\bv_h):=\nu \sum_{T \in \CT_h} \int_{T} \nabla \bu_h:\nabla \bv_h\,dx,\;
b(\bv_h,q_h):=\sum_{T \in \CT_h} \int_{T}\div \bv_h\,q_h \,dx.
\end{eqnarray*}
In order to show that the saddle point problem 
\eqref{mixed} has a unique solution, we want to apply a
standard saddle point theory \cite{BF91,BS94}. 
To this end, we need to show the following 
{\em three conditions of well-posedness.}
\begin{enumerate}
\item The linear form $\ell(\cdot)$,  
the bilinear forms $a(\cdot,\cdot)$ 
and $b(\cdot,\cdot)$ are continuous on the spaces on which 
they are defined. 
\item The bilinear form 
$a(\cdot,\cdot)$ is coercive on the space $K$ defined as 
\[ K =\{ \bv_h \in \bV_h:\,
b(\bv_h,q_h) =  0,\, q_h \in P_h\}.\]
\item The finite element pair $(\bV_h,P_h)$ satisfies a 
uniform inf-sup condition. That means 
the bilinear form $b(\cdot,\cdot)$ satisfies 
\begin{eqnarray}\label{infsupc} 
\sup_{\bv_h \in \bV_h} \frac{b(\bv_h,q_h)}{ 
\|\bv_h\|_{1,h} } \geq \beta \|q_h\|_{0,\Omega},\quad q_h \in P_h
\end{eqnarray}
for a constant $\beta >0$ independent of  the mesh-size $h$.
\end{enumerate}
It is standard that the linear form $\ell(\cdot)$, and 
the bilinear forms $a(\cdot,\cdot)$ and $b(\cdot,\cdot)$
are continuous, and $a(\cdot,\cdot)$  is coercive due to 
Poincar\'e inequality.  It remains to prove that 
the finite element pair $(\bV_h,P_h)$  satisfies the 
uniform inf-sup condition. 

A finite element method for the Stokes equations is presented in 
\cite{LC08} using the finite element pair $(\bV_h,P_h)$, where the 
authors use  a local stabilization to obtain the stability of the system. 
Here we show that the stabilization is not necessary, and 
the finite element pair $(\bV_h,P_h)$ also satisfies 
the uniform inf-sup condition. This is the main goal of the paper. 

To show this we use the 
fact that the finite element pair 
$(\bV_h,Q_h)$ satisfies a uniform inf-sup condition, which is proved
by Crouzeix and Raviart 
\cite{CR73}. Hence there exists a constant $\hat \beta >0$ indpendent
of the mesh-size $h$ such that 
\[ \sup_{\bv_h \in \bV_h} \frac{b(\bv_h\,q_h)}{ 
\|\bv_h\|_{1,h} } \geq \hat \beta \|q_h\|_{0,\Omega},\quad q_h \in Q_h.
\]
 We now introduce another finite element space 
\[ R_h := \{ q_h \in Q_h:\, q_h|_{S_i} =P_0(S_i),\; i \in \CN_h\},\]
where $S_i$ is the support of the standard linear basis function  
$\phi_i  \in P_h$ associated with the vertex $\bx_i\in \CV_h$.  
We use this space to show that the finite element pair $(\bV_h,P_h)$  satisfies the 
uniform inf-sup condition.  Since $R_h \subset Q_h$,
the finite element pair $(\bV_h,R_h)$ also satisfies the 
uniform inf-sup condition. Thus we can get a consant 
$\tilde \beta >0$ independent of the mesh-size $h$ such that 
\[ \sup_{\bv_h \in \bV_h} \frac{b(\bv_h,q_h)}{ 
\|\bv_h\|_{1,h} } \geq \tilde\beta \|q_h\|_{0,\Omega},\quad q_h \in R_h.
\]

In order to show that the finite element pair $(\bV_h,P_h)$ 
satisfies the uniform inf-sup condition we introduce an interpolation 
operator $I_h$ \cite{RTY02,Lam08a,Lam09a} 
\[ I_h : P_h \rightarrow R_h\] 
defined as  \[I_hq_h =  \sum_{i=1}^N q_i\chi_i\] 
for  $q_h= \sum_{i=1}^N q_i \phi_i \in P_h$, where 
$\chi_i$ is a characteristic function of the set $S_i$, and $\phi_i$
is the standard linear hat function associated with  the vertex $\bx_i$ for $ i \in
\CN_h$.  We note that the interpolation operator $I_h$ is
well-defined, and it is not the standard Fortin interpolation operator.
We define an element-wise defined divergence $\divh \bv_h $
of a vector function 
$\bv_h \in \bV_h$ 
\[ \divh \bv_h|_{T} = \div \bv_h|_{T},\quad T \in \CT_h.\]
We have now following two lemmas. In the following, we use generic constants $C$, $C_1$ and $C_2$. They
may  take different values at different places but they 
do not depend on the mesh-size $h$.
\begin{lemma} \label{lemma1}
Given $\bv_h \in \bV_h$ and $q_h \in P_h$ 
we have \[
(d+1) \int_{\Omega} \divh\bv_h\,q_h\,dx = \int_{\Omega} \divh\bv_h\,I_hq_h\,
dx. \]
\end{lemma}
\begin{proof}
Let $q_h = \sum_{i=1}^N q_i \phi_i$, and hence 
$I_hq_h = \sum_{i=1}^N q_i \chi_i$. Then 
using the fact that $\divh \bv_h$ is piecewise constant with 
respect to the mesh $\CT_h$  we have 
\[  \int_{\Omega} \divh\bv_h\,I_hq_h\,dx = 
\sum_{i=1}^N q_i
\int_{S_i} \divh\bv_h\,dx = \sum_{i=1}^N q_i
\sum_{T \subset S_i}|T| (\div\bv_h)|_{T}.\]
Similarly, 
\[ 
\int_{\Omega} \divh\bv_h\,q_h\,
dx = \sum_{i=1}^N q_i \int_{S_i} \divh\bv_h\,\phi_i\,dx
= \frac{1}{d+1} \sum_{i=1}^N q_i \sum_{T \subset S_i}|T| (\div\bv_h)|_{T}.
\]

\end{proof}

\begin{lemma}\label{lemma2}
There exist positive constants $C_1$ and $C_2$ with 
\begin{equation} \label{eq2}
C_1 \|q_h\|_{0,\Omega} \leq \|I_h q_h\|_{0,\Omega} \leq C_2 \|q_h\|_{0,\Omega}.
 \end{equation} 
\end{lemma}

\begin{proof}
The proof follows by using the  fact that $\|I_hq_h\|^2_0$ and $\|q_h\|^2_0$ and 
$\sum_{i=1}^Nq^2_i h^2_i$ are equivalent, where $h_i$ is the local 
mesh-size at the $i$-th node of $\CT_h$.
\end{proof}
We are now in a position to prove the main result of the paper. 
\begin{theorem}
The finite element pair 
 $(\bV_h, P_h)$ satisfies the inf-sup condition. That means there
 exists a consant $ \beta>0$ independent of mesh-size $h$ such that 

\[ \sup_{\bv_h \in \bV_h} \frac{b(\bv_h,q_h)}{ 
\|\bv_h\|_{1,h} } \geq \beta \|q_h\|_{0,\Omega},\quad q_h \in P_h.
\]
\end{theorem}
\begin{proof}
The proof is straightforward. Let $q_h \in P_h$, then $I_hq_h \in
R_h$. 
Since the pair $(\bV_h,  R_h)$ satisfies the 
inf-sup condition, we can find an element $\bv_h \in \bV_h$ satisfying 
\[\int_{\Omega} \divh\bv_h\,I_hq_h\,dx \geq C \|I_hq_h\|^2_{0,\Omega},\quad 
\|\bv_h\|_{1,h} \leq C \|I_hq_h\|_{0,\Omega}.
\]
Hence the results of Lemma \ref{lemma1} and \ref{lemma2} yield 
\[ \int_{\Omega} \divh\bv_h\,I_hq_h\,dx 
= (d+1) \int_{\Omega} \divh\bv_h\,q_h\,dx \geq 
C \|I_hq_h\|^2_{0,\Omega} \geq 
C \|q_h\|^2_{0,\Omega},\]
and $\|\bv_h\|_{1,h} \leq C \|q_h\|_{0,\Omega}.$
Hence we have a constant $C$ independent of the mesh-size $h$ such
that 
\[ \sup_{\bv_h \in \bV_h} \frac{b(\bv_h,q_h)}{ 
\|\bv_h\|_{1,h} } \geq C \|q_h\|_{0,\Omega},\quad q_h \in P_h.
\]
\end{proof}

\section{Numerical experiments}\label{sec:num}
In this section we present two numerical experiments for the proposed 
finite element scheme. For the numerical experiment we consider a 
simple unit square $\Omega = (0,1)^2$, and  a uniform 
initial triangulation consisting of eight triangles. 
\paragraph{First example.}
For the first example we consider an exact solution 
given in \cite{BDG06}, where the exact solution for the velocity 
$\bu = (u_1,u_2)$ is given by 
\[ u_1 = x+{x}^{2}-2\,xy+{x}^{3}-3\,x{y}^{2}+{x}^{2}y,\quad 
u_2 = -y-2\,xy+{y}^{2}-3\,{x}^{2}y+{y}^{3}-x{y}^{2},\]
and the exact solution for the pressure is given by 
\[ p = xy+x+y+{x}^{3}{y}^{2}-\frac{4}{3}.\]
Using the kinematic viscosity $\nu =1$ and the exact solution 
we compute the right-hand side function $\fb$ and 
the Dirichlet boundary condition for the velocity.
Here we compute the errors in 
the velocity and the pressure approximation using the broken $H^1$- semi-norm and 
the $L^2$- norm, respectively. The numerical results are tabulated in 
Table \ref{exa1}. From the presented table we can see 
the optimal convergence of the 
velocity approximation in the broken $H^1$- semi-norm, and 
a super-convergence result for the pressure in the 
$L^2$-norm. As we expect a convergence rate of order 
$1$  for the pressure approximation in the $L^2$-norm but 
get a better approximation of order $1.5$, 
this is a super-convergence. This 
better convergence is due to the fact that 
 we have used the standard continuous linear finite element space 
 for the pressure approximation.

\begin{table}[!htb]
\caption{Discretization errors for the velocity and pressure, Example 1}
\begin{center}
\begin{tabular}{|c|c|c|c|c|c|c|c|c|}\hline
level $l$ & \# elem. & \multicolumn{2}{c|}{$|u -u_h|_{1,h}$}
 & \multicolumn{2}{c|}{$\|p -p_h\|_{0,\Omega}$} \\ \hline
 0 &     8 & 2.53070e-01 &  & 9.75595e-02 &  \\\hline
 1 &    64 & 1.32989e-01 & 0.93 & 3.35584e-02 & 1.54  \\\hline
 2 &   512 & 6.78573e-02 & 0.97 & 1.14072e-02 & 1.56  \\\hline
 3 &  4096 & 3.42262e-02 & 0.99 & 3.87893e-03 & 1.56  \\\hline
 4 & 32768 & 1.71804e-02 & 0.99 & 1.32844e-03 & 1.55  \\\hline
 5 & 262144 & 8.60590e-03 & 1.00 & 4.58976e-04 & 1.53  \\\hline
\end{tabular}
\end{center}
\label{exa1}
\end{table}
\paragraph{Second example.}
For our second example we consider the same computational domain but a
different exact solution with kinematic viscosity $\nu =5$. 
Here the exact solution for the velocity 
$\bu = (u_1,u_2)$ is given by 
\[ u_1 ={{\rm e}^{-x+y}}\sin \left( 5\,x \right) ,\quad 
u_2 = {{\rm e}^{-x+y}}\sin \left( 5\,x \right) -5\,{{\rm e}^{-x+y}}\cos
 \left( 5\,x \right) ,\]
and the exact solution for the pressure is given by 
\[ p = xy \left( 1-x \right)  \left( 1-y \right) -\frac{1}{36}.\]
As in the first example, the source term $\fb$ and 
the Dirichlet boundary condition are computed by 
using the chosen exact solution.  We have shown 
the errors in the velocity and pressure approximating 
using norms as in the previous example in 
Table \ref{exa2}. 
We can see that the velocity approximation converges 
linearly to the exact solution, whereas 
the pressure approximation shows a 
super-convergence as in the first example.

\begin{table}[!htb]
\caption{Discretization errors for the velocity and pressure, Example 2}
\begin{center}
\begin{tabular}{|c|c|c|c|c|c|c|c|c|}\hline
level $l$ & \# elem. & \multicolumn{2}{c|}{$|u -u_h|_{1,h}$}
 & \multicolumn{2}{c|}{$\|p -p_h\|_{0,\Omega}$} \\ \hline
 0 &     8 & 4.85293e+00 & & 3.78781e+00 &  \\\hline
 1 &    64 & 2.57400e+00 & 0.91 & 1.36397e+00 & 1.47  \\\hline
 2 &   512 & 1.31776e+00 & 0.97 & 4.77524e-01 & 1.51  \\\hline
 3 &  4096 & 6.65496e-01 & 0.99 & 1.66988e-01 & 1.52  \\\hline
 4 & 32768 & 3.34230e-01 & 0.99 & 5.85231e-02 & 1.51  \\\hline
 5 & 262144 & 1.67459e-01 & 1.00 & 2.05565e-02 & 1.51  \\\hline
\end{tabular}
\end{center}
\label{exa2}
\end{table}
\section{Conclusion}\label{sec:con}
We have presented a finite element method for the Stokes equations
using lowest order Crouzeix-Raviart element for the velocity and 
continuous linear element for the pressure. 
We have proved the fact the stabilization is not necessary as 
proposed in \cite{LC08}. The numerical experiments 
support the theoretical results. 

{\normalsize 
\bibliographystyle{plain}
\bibliography{total}
}
\end{document}